\documentclass[10pt]{amsart}
\usepackage{amsmath}
\usepackage{latexsym}

\usepackage{amssymb}
\usepackage{hhline}
\usepackage{array}
\usepackage{tabularx}
\usepackage{longtable}
\usepackage{xcolor}

\theoremstyle{remark}
\newtheorem{remark}{Remark}
\theoremstyle{plain}

\newtheorem{lemma}{Lemma}[section]
\newtheorem{theorem}{Theorem}

\numberwithin{equation}{section}

\theoremstyle{definition}
\newtheorem{definition}{Definition}

\def\ds{\displaystyle}
\def\({\left(}
\def\){\right)}
\def\pfrac#1#2{\left(\frac{#1}{#2}\right)}
\def\a{\alpha}

\def\eps{\varepsilon}

\def\lam{\lambda}
\def\sg{\sigma}
\def\sit{\sigma + it}
\def\nts{\negthickspace}
\newcommand{\be}{\begin{equation}}
\newcommand{\ee}{\end{equation}}

\newcommand{\rev}[1]{{\color{red}#1}}

\begin{document} 

\title{Zero-free regions for the Riemann zeta function} 
\author{Kevin Ford}
\begin{abstract}
We prove explicit zero-free regions for the Riemann zeta function.
{\rev{Corrections to the published version highlighted in red.}}
\end{abstract}

\thanks{Research supported in part by National Science Foundation grant
DMS-0070618.}

\address{Department of Mathematics, University of South Carolina,
Columbia, SC 29208,  USA} 
 
\maketitle 

\section{Introduction}\label{sec:intro}               %

The methods of Korobov \cite{Ko} and Vinogradov \cite{V}
produce a zero-free region
for the Riemann zeta function $\zeta(s)$ of the following strength: for
some constant $c>0$, there
are no zeros of $\zeta(s)$ for $s=\beta+it$ with $|t|$ large and
\be
1 - \beta \le \frac{c}{(\log |t|)^{2/3}(\log\log |t|)^{1/3}}.
\label{eq:VK}
\end{equation}
The principal tool is an upper bound for $|\zeta(s)|$ near the line
$\sigma=1$.  One form of this upper bound was given by Richert \cite{Ri} as
\be
|\zeta(\sigma+it)| \le A |t|^{B(1-\sigma)^{3/2}} \log^{2/3} |t|
\qquad (|t|\ge 3, \tfrac12 \le \sigma \le 1) \label{eq:Richert}
\end{equation}
with $B=100$ and $A$ and unspecified absolute constant.
Subsequently, \eqref{eq:Richert} was proved with smaller values of $B$,
the best published value
being 18.497 \cite{Ku} (the author has a new
result \cite{F} that \eqref{eq:Richert} holds with $B=4.45$, $A=76.2$). 

Table 1 shows the historical progression of zero-free regions
for $\zeta(s)$ prior to the work of Vinogradov and Korobov. 
\begin{table}[h]
\setlength{\extrarowheight}{10pt}
\begin{tabular}{|c|c|} \hline
Zero-free region & Reference \\ \hline
$\ds 1-\beta \le  \frac{c}{\log |t|}$ & 
de la Vall\'ee Poussin \cite{VP}, 1899 \\
$\ds 1-\beta \le  \frac{c\log\log |t|}{\log |t|}$ & Littlewood \cite{L},
 1922 \\
$\ds 1-\beta \le \frac{c}{(\log |t|)^{3/4+\eps}}$ &
Chudakov \cite{Chu}, 1938 \\
\hline
\end{tabular}
\smallskip
\caption{}
\end{table}
\bigskip

Recently,  versions of \eqref{eq:VK} with explicit constants $c$ have been 
given, valid for $|t|$ sufficiently large.
Popov \cite{P} showed that \eqref{eq:VK}
holds with $c=0.00006888$.  Heath-Brown \cite{HB1} proved \eqref{eq:VK} with
$c \approx 0.0269 B^{-2/3}$, and he noted (but did not give details)
that the methods
of \cite{HB2} could be used to improve $0.0269$ to about $0.0467$.
The main object of this note is to
improve the constant $c$ as a function of $B$.

\begin{theorem} \label{thm1}
If~\eqref{eq:Richert} holds with a certain constant $B$, then
for large $|t|$, $\zeta(\beta+it) \ne 0$ for
$$
1-\beta \le \frac{0.05507 B^{-2/3}}{(\log |t|)^{2/3} (\log\log |t|)^{1/3}}.
$$
\end{theorem}

Taking $B=4.45$ (from \cite{F}) in Theorem~\ref{thm1}
 gives the zero-free region \eqref{eq:VK} with
$c=\frac{1}{49.13}$.
In addition, we prove a totally explicit zero-free region
of type \eqref{eq:VK}, with an explicit $c$ and valid for all $|t|\ge 3$.
This depends on both $A$ and $B$ in \eqref{eq:Richert}), and 
may be used to give completely explicit
bounds for prime counting functions (see e.g. \cite{RS1}, \cite{RS2}, 
\cite{RR}).  Cheng \cite{Ch2} proved \eqref{eq:Richert} with
$A=175$ and $B=46$ and used this to deduce that
\eqref{eq:VK} holds for all $|t| \ge 3$ with the
constant $c=1/990$. 

\begin{theorem}\label{thm2}
\rev{Suppose \eqref{eq:Richert} holds with $0<B<10$ and $A \ge 1$.}  Suppose that
$T_0 \ge e^{30000}$
and $\frac{\log T_0}{\log\log T_0} \ge \frac{1740}{B}$.
Suppose the zeros $\beta+it$ of $\zeta(s)$ with $T_0-1\le t\le T_0$
all satisfy
\begin{equation}
1-\beta \ge  \frac{M_1 B^{-2/3}}{(\log t)^{2/3}
 (\log\log t)^{1/3}}, \label{eq:thm2}
\end{equation}
where 
$$
M_1 = \min \( 0.05507, \frac{0.1652}{2.9997+
\operatornamewithlimits{max}_{t\ge T_0}
X(t)/\log\log t} \),
$$
and
\[
X(t) = \rev{1.1582}\log A+\rev{0.8332}+\rev{0.234}\log \(\!\tfrac{B}{\log\log t}\!\) 
+  \( \rev{\tfrac{1.239}{B^{4/3}}-\tfrac{2.065}{B^{1/3}}} \)
\( \tfrac{\log\log t }{\log  t} \)^{\frac13}.
\]
Then \eqref{eq:thm2} is satisfied for all zeros with $t \ge T_0$.
\end{theorem}

Since $\frac{0.1652}{2.9997} > 0.05507$ \rev{ and $X(t)$ is bounded above},
 $M_1 = 0.05507$ when $T_0$ is sufficiently large.
By classical zero density bounds (see e.g. Chapter 9 of \cite{T}), for some
positive $\delta$, the number of zeros of $\zeta(s)$ is the rectangle
$\frac34 \le \Re s \le 1, 0 < \Im s \le T$ is $O(T^{1-\delta})$.  Thus
for most $T_0$, $\zeta(s)$ is zero free in the region $\frac34 \le \Re s
\le 1, T_0-1\le \Im s \le T_0$.  Taking such $T_0$ which is sufficiently
large, we see that Theorem \ref{thm1} follows from Theorem \ref{thm2}.

To prove a totally explicit zero-free region of type \eqref{eq:VK}
 for $|t|\ge 3$, we make use of
classical type (de la Val\'ee Poussin type) zero-free regions
 for smaller $|t|$. These take the form
\be
1-\beta \le \frac{c}{\log |t|} \qquad (|t| \ge 3). \label{eq:clas zfr}
\end{equation}
Stechkin \cite{S} proved \eqref{eq:clas zfr} with $c=1/9.646$ 
(he rounded this to $c=9.65$ in his Theorem 2).
Very tiny refinements were subsequently given by
Rosser and Schoenfeld \cite{RS2} and by Ramar\'e and Rumely \cite{RR}.
With an explicit version of van der Corput's bound
$|\zeta(1/2+it)| \ll |t|^{1/6}\log |t|$ for $|t|\ge 3$,
the methods of this paper produce a zero-free region
\be
1-\beta \le \frac{1}{C_1(\log |t|+6\log\log|t|)+C_2}, \qquad (|t| \ge 3),
\label{eq:new clas}
\end{equation}
with $C_1\approx 3.36$ and an explicit $C_2$.
Better upper bounds are known for $|\zeta(1/2+it)|$ for large $t$,
the best being $O_\eps(|t|^{89/570+\eps})$ due to Huxley \cite{H}.
The implied constants are too large to improve the zero-free region,
however.
The zero-free region \eqref{eq:new clas}
also follows from Heath-Brown's methods with the same $C_1$ (and slightly
larger $C_3$).  In fact,
the methods of this paper do not improve on Heath-Brown's methods when it comes
to classical type zero-free regions for $\zeta(s)$ or zero-free regions
for Dirichlet $L$-functions $L(s,\chi)$ when $|t|$ is small and the
conductor of $\chi$ is large (e.g. those in \cite{HB2}).  Our methods do
improve the Vinogradov-Korobov zero-free regions for $L(s,\chi)$ when the
conductor of $\chi$ is fixed and $|t|$ becomes large.

It is known \cite{LRW} that all zeros with $|\Im \rho| \le
5.45 \times 10^8$ in fact lie on the critical line.  Still, at $t=5.45\times
10^8$, $6\log\log t \approx 0.895\log t$, so improving greatly on
Stechkin's region for all $|t|\ge 3$ with \eqref{eq:new clas} is
not possible.  Still, we can make a modest improvement using the bound
\be
|\zeta(1/2+it)| \le \min\( 6t^{1/4}+57, 3t^{1/6}\log t \) \qquad (t\ge 3).
\label{eq:CG}
\end{equation}
\rev{claimed} by Cheng and Graham \cite{CG}.
\footnote{\rev{The proof in \cite{CG} of \eqref{eq:CG} contains an unfixable error, namely Lemma 3 is
  false (the best possible estimate in the Lemma was proved by Landau in 1927).
  Since the original version of this paper was published in 2002, I became aware of
  an older bound $|\zeta(1/2+it)| \le 4(|t|/2\pi)^{1/4}$ for $|t| \ge 128\pi$
  of Lehman \cite{Le}, Lemma 2.
  This is better than
  the first bound in \eqref{eq:CG} and by itself leads to better numerical bounds in Theorem 4.  
  More recently,
  Trudgian and Hiary have published claimed improvements to the second bound in \eqref{eq:CG},
  although both papers also make critical use the
  erroneous Lemma 3 from \cite{CG}.   For details of the error and
  how to correct it, see the papers by
   Patel \cite{Pa} and 
 Hiary, Patel, Yang \cite{HPY}, the latter proving the bound
$|\zeta(1/2+it)| \le 0.618 |t|^{1/6}\log |t|$ for $|t| \ge 3$.}}

\begin{theorem}\label{thm3}  Let $T_0=5.45\times 10^8$ and let
\be
J(t) = \min \( \tfrac14 \log t +1.8521, \tfrac16 \log t +\log\log t +\log 3 \),
\label{eq:J(t)}
\end{equation}
\rev{and}
\[
\rev{c_6(t) = 0.05635 \, \frac{J(4t+1)-J(t)+c_7(L_2-\log\log t)}{J(t)+0.155\log\log t+0.675},}
\]
Then $\zeta(\beta+it)\ne 0$ for $t\ge T_0$ and 
\be
1-\beta \le \rev{\frac{0.04962-c_6(t)}{J(t)+0.155\log\log t + 0.675}.}
\label{eq:clas 1}
\end{equation}
\end{theorem}

\rev{The right side multiplied by $\log t$ is decreasing in $t$.}
Therefore, we conclude as a corollary that

\begin{theorem}\label{thm4}
We have $\zeta(\beta+it) \ne 0$ for $|t|\ge 3$ and
$$
1-\beta \le \frac{1}{\rev{8.464} \log |t|}.
$$
\end{theorem}

Further verification that the zeros of $\zeta(s)$ for some range of
$t>5.45\times 10^8$ would give an improved constant in Theorem \ref{thm4},
as would an improvement in the bound for $|\zeta(1/2+it)|$ in
the vicinity of $t=T_0$.

We now return to the problem of producing a totally explicit
 zero-free regions of Korobov-Vinogradov type.
Taking $B=4.45$, $A=76.2$ (from \cite{F}), \rev{and writing $\tau=4t+1$}, we find
\rev{for $t\ge e^{54000}$ that}
\begin{align*}
\rev{1.1582}\log A&+\rev{0.8332}+\rev{0.234\log \(\tfrac{B}{\log\log \tau}\)} +
\rev{\( \tfrac{1.239}{B^{4/3}}-\tfrac{2.065}{B^{1/3}} \)
 \( \tfrac{\log\log \tau }{\log  \tau} \)^{\frac13}} \\
&\le \rev{6.2015}-\rev{0.234}\log\log\log t - \rev{1.0861}\pfrac{\log\log t}
  {\log t}^{1/3} \\
&\le \rev{5.5789},
\end{align*}
\rev{as the middle expression is decreasing in $t$ for $\log t \ge 54000$.}
Thus
\be
M_1 \ge \min \( 0.05507, \frac{0.1652}
{2.9997 +\frac{\rev{5.5789}}{\log\log T_0}} \). \label{eq:M1}
\end{equation}
We take $T_0=e^{54550}$, use Theorem \ref{thm3}
for $t \le T_0+1$, and Theorem
\ref{thm2} plus \eqref{eq:M1} for larger $t$.  This gives

\begin{theorem}\label{thm5}
The function $\zeta(\beta+it)$ is nonzero in the
region
$$
1-\beta \le \frac{1}{57.54(\log |t|)^{2/3} (\log\log |t|)^{1/3}},
 \quad |t|\ge 3.
$$
\end{theorem}

\section{The zero detector} \label{sec:zero detect}   %

\begin{lemma}\label{lem:entire}
Suppose $f$ is the quotient of two entire 
functions of order $<k$, where $k$ is a positive integer, and $f(0)\ne 0$.
 If $z$ is neither a pole nor a zero of $f$, then
\begin{align*}
\frac{f'(z)}{f(z)} &= \sum_{|\rho| \le 2|z|}  \frac{(z/\rho)^{k-1}}{z-\rho}
 m_\rho + O_f \(|z|^{k-1} \), \\
\bigl|\log |f(z)| \bigr| &\le \left| \sum_{|\rho| \le 2|z|} 
\log\left| (1-z/\rho)e^{g(z/\rho)} \right|  \right| + O_f \( |z|^k \),
\end{align*}
where $\rho$ runs over the zeros and poles of $f$ (with multiplicity),
$g(y) = y + \frac12 y^2 + \cdots + \frac{1}{k-1} y^{k-1}$,
 and $m_\rho$ is either 1 (if $\rho$ is a zero of $f$)
or $-1$ (if $\rho$ is a pole of $f$).  The implied constants depend on $f$.
\end{lemma}

\begin{proof}  By theorems of Weierstrass and Hadamard (\cite{SZ}, 
Ch. VII, (2.13) and (10.1)),
$$
f(z) = e^{f_1(z)} \prod_\rho \left[ (1-z/\rho) e^{g(z/\rho)} \right]^{m_\rho},
$$
where $f_1$ is a polynomial of degree $\le k$.  Therefore, assuming that $z$
is not a zero or pole of $f$, we have
\begin{align*}
\log |f(z)| &= \Re f_1(z) + \sum_\rho m_\rho 
\( \log |(1-z/\rho) e^{g(z/\rho)}|\), \\
\frac{f'(z)}{f(z)} &= f_1'(z) + \sum_\rho m_\rho \( \frac{1}{z-\rho} +
 \frac{1}{\rho} +\frac{z}{\rho^2} + \cdots + \frac{z^{k-2}}{\rho^{k-1}} \).
\end{align*}
Now suppose $|\rho| > 2|z|$.  We then have 
$$
\left|  \frac{1}{z-\rho} +
 \frac{1}{\rho} +\frac{z}{\rho^2} + \cdots + \frac{z^{k-2}}{\rho^{k-1}} \right|
= \left| \frac{(z/\rho)^{k-1}}{z-\rho} \right| \le 2\frac{|z|^{k-1}}{|\rho|^k}.
$$
Since $\sum_\rho 1/|\rho|^k$ converges, the first part of the lemma follows.
Similarly
$$
\left| (1-z/\rho) e^{g(z/\rho)} \right| \le \exp \{ \tfrac{2}{k} | \tfrac{z}
{\rho} |^k \},
$$
and the second part follows.
\end{proof}

The next lemma is the main ``zero detector''.  Instead of integrating
around a small circle centered at $z=z_0$ (as in \cite{HB2}, Lemma 3.2),
we integrate over two vertical lines.

\begin{lemma}\label{lem:zero detect}
Suppose $f$ is the quotient of two entire 
functions of finite order, and does not have a zero or a pole at $z=z_0$ nor
at $z=0$.
Then, for all $\eta> 0$ except for a set of Lebesgue measure 0 (the exceptional
set may depend on $f$ and $z_0$), we have
\begin{multline*}
-\Re \frac{f'(z_0)}{f(z_0)} = \frac{\pi}{2\eta}\sum_{|\Re 
(z_0-\rho)| \le \eta} m_\rho \Re \cot \pfrac{\pi(\rho-z_0)}{2\eta} \\
+ \frac{1}{4\eta} 
\int_{-\infty}^{\infty} \frac{\log \left| f \(z_0-\eta+
\tfrac{2\eta i u}{\pi} \) \right| -
\log \left| f\(z_0+\eta+\tfrac{2\eta i u}{\pi} \) \right|}{\cosh^2 u} du,
\end{multline*}
where $\rho$ runs over the zeros and poles
of $f$ (with multiplicity), and $m_\rho$ is either 1 (if $\rho$ is a zero of
 $f$) or $-1$ (if $\rho$ is a pole of $f$).
\end{lemma}

\begin{proof} We must exclude $\eta$ for which the lines $\Re z = z_0 
\pm \eta$ come ``too close'' to a zero or pole of $f$, since otherwise
 the above integral might not
converge.  By hypothesis, for some integer $k$, $f$ is the quotient of two
entire functions of order $<k$.  We say a positive real number $\eta$
 is ``good'' if there is a positive number $\delta$ such that
 for every zero/pole $\rho$ of $f$, $|\Re (\rho-z_0) \pm \eta| \ge
 \delta |\rho|^{-k}$.
The number $\delta$ may depend on $\eta$.  Since $\sum_\rho |\rho|^{-k}$ 
converges, the set of $\eta$ for which $|\Re (\rho-z_0)\pm \eta| \le \delta
|\rho|^{-k}$ has measure $O(\delta)$ (here and throughout this proof,
 implied constants depend on $f$ and $z_0$). 
Taking $\delta\to 0$ shows that the measure of ``bad'' $\eta$ is 0.

Suppose now that $\eta$ is ``good'' with an associated number $\delta$.
We may assume that $0<\delta\le 1$.  Let $T$
be a large real number such that $T\ge \eta$, $T\ge 2|z_0|$ and
for all zeros/poles $\rho$ of $f$, $|\Im (\rho-z_0) \pm T|\ge |\rho|^{-k}$. 
Since $\sum_\rho |\rho|^{-k}$ converges,
 the set of ``bad'' $T$ has measure $O(1)$.  Consider the
contour $C=C_1 \cup C_2 \cup C_3 \cup C_4$, where the $C_j$ are the line
segments connecting the points $\eta-iT, \eta+iT, -\eta+iT, -\eta-iT, \eta-iT$,
respectively.  Let
$$
I=I_1+I_2+I_3+I_4, \quad I_j = \rev{\frac{1}{2\pi i}} \int_{C_j} \frac{f'(z+z_0)}{f(z+z_0)}h(z)\, dz,
$$
where
$$
h(z) = \frac{\pi}{2\eta} \cot \( \frac{\pi z}{2\eta} \).
$$
By Cauchy's Residue Theorem,
\begin{equation}
I = \frac{f'(z_0)}{f(z_0)} + \sum_{\substack{|\Re (\rho-z_0) | \le \eta \\
|\Im (\rho-z_0)|\le T}} m_\rho h(\rho-z_0). \label{eq: I}
\end{equation}
There is a holomorphic branch of $\log f(z+z_0)$ on $C^*$, the contour $C$
cut at the point $\eta$.  Applying integration by parts, and noting that
$h(\eta)=0$, we have
\begin{equation}
\begin{split}
I &= \lim_{\varepsilon \to 0^+} \left[ h(z) \log f(z+z_0) \right]
_{\eta+i\epsilon}^{\eta-i\epsilon} - \frac{1}
{2\pi i} \int_{C^*} h'(z) \log f(z+z_0)\, dz \\
&= - (J_1+J_2+J_3+J_4), \qquad J_j = \frac{1}{2\pi i} \int_{C_j}h'(z)
\log f(z+z_0)\, dz.
\end{split}
\label{eq: I2}
\end{equation}
The number of zeros/poles $\rho$ with $|\rho| \le x$ is $O(x^k)$, and 
$|\rho| \gg 1$ for every $\rho$.
By our assumptions about $T$, when $z\in C$ we have $|z+z_0| \ll
T$.  Therefore, by Lemma \ref{lem:entire} and our assumption about $\eta$,
\begin{align*}
\left|\frac{f'(z+z_0)}{f(z+z_0)} \right| &\ll
T^{k-1} +  \sum_{|\rho| \le 2\rev{|z+z_0|}} \frac{|(z+z_0)/\rho|^{k-1}}{|z+z_0-\rho|}\\
&\ll T^{k-1} + \rev{T^k} \, \frac{T^{k-1}}{|\rho|^{k-1}} \frac{|\rho|^k}{\delta}
 \ll \delta^{-1} \rev{T^{2k}}.
\end{align*}
Likewise, using the second part of Lemma 2.1, 
$$
|\log |f(z+z_0)|| = O(\rev{T^k}+T^k\log(T \delta^{-1}))
$$
for $z\in C$.  Thus, there is a branch of $\log f(z+z_0)$ with 
$$
|\log f(z+z_0)| \ll T^{2k}\delta^{-1}.
$$
This is important to the estimation of $J_2$ and $J_4$.   Since 
$$
h'(z)=-\frac{\pi^2}{4\eta^2} \csc^2\(\frac{\pi z}{2\eta} \),
$$
we have $|h'(\eta \pm iT)| \ll \eta^{-2} e^{-\pi T/(2\eta)}$.
Therefore, $|J_2|+|J_4| \to 0$ as $T\to \infty$.
  Parameterizing the line segments
$C_1$ and $C_3$ with $z=\pm \eta + \frac{2\eta i u}{\pi}$ and taking real
parts gives 
$$
\Re (J_1+J_3) = \frac{1}{4\eta} \int_{-\frac{\pi T}{2\eta}}^{\frac{\pi T}
  {2\eta}} \frac{\log \left| f \(z_0-\eta+
  \tfrac{2\eta i u}{\pi} \) \right| -
\log \left| f\(z_0+\eta+\tfrac{2\eta iu}{\pi} \) \right|}{\cosh^2 u} du.
$$
Recalling \eqref{eq: I} and \eqref{eq: I2},
 this proves the lemma upon letting $T\to \infty$.
\end{proof}

%
\section{Bounds for $\zeta(s)$} \label{sec:zeta bounds}
%

\begin{lemma}\label{lem:zeta basic}
\rev{(i) For all $\sigma>1$ and real $t$,
$$
\frac{1}{\zeta(\sigma)} \le |\zeta(\sigma+it)| \le \zeta(\sigma)
$$
and
$$
\left| - \frac{\zeta'}{\zeta}(\sigma+it)\right| < \frac{1}{\sigma-1}.
$$
(ii) When $1<\sigma \le 1.8$,
$$
\zeta(\sigma) \le 0.64 + \frac{1}{\sigma-1}.
$$
}
\end{lemma}

\begin{proof}
For the first line of inequalities \rev{in (i)}, we start with
$$
|\zeta(\sigma+it)| \le \sum_{n=1}^\infty n^{-\sigma}= \zeta(\sigma)
$$
and similarly
$$
|\zeta(\sigma+it)|^{-1} = \left| \sum_{n=1}^\infty \mu(n) n^{-\sigma-it}
\right|  \le \sum_{n=1}^\infty n^{-\sigma}= \zeta(\sigma).
$$
The \rev{second line} 
 follows from $|-\frac{\zeta'}{\zeta}(\sigma+it)| \le
-\frac{\zeta'}{\zeta}(\sigma)$ and
$$
-\zeta'(\sg)=\sum_{n=1}^\infty \( \sum_{m\ge n+1} m^{-\sigma} \)
 \log \pfrac{n+1}{n} <
\sum_{n=1}^\infty \frac{n^{1-\sg}}{\sg-1} \, \frac1{n} = \frac{\zeta(\sg)}
{\sg-1}.\quad\text{}
$$
Next, since $x^{-\sigma}$ is convex, we have
\begin{align*}
\zeta(\sigma) &\le 1 + \rev{\frac{1}{2^\sigma} + \int_{5/2}^\infty \frac{du}{u^\sigma} = 1 +
\frac{1}{2^\sigma} + \frac{(5/2)^{-(\sigma-1)}}{\sigma-1}} \le \rev{0.64} + \frac{1}{\sigma-1}
\end{align*}
\rev{by a short calculation.}
In fact, near $\sigma=1$ we have $\zeta(\sigma)=\frac{1}{\sigma-1}+\gamma+
O(\sigma-1)$, where $\gamma=0.5772\ldots$ is the Euler-Mascheroni constant
(see e.g. \cite{T}, (2.1.16)).
\end{proof}

%
%

\begin{lemma}\label{lem:zeta on -1/2}
 For real $u$,
\[
\left| \frac{\zeta'(-\tfrac12+iu)}{\zeta(-\tfrac12+iu)} \right|
\le 4.62+ \frac12 \log(1+u^2/9).
\]
\end{lemma}

\begin{proof} 
\rev{
By the functional equation for $\zeta(s)$ (cf. \cite{D}, Ch. 12, (8)--(10)), 
\[
-\frac{\zeta'(w)}{\zeta(w)} = \frac{\zeta'(1-w)}{\zeta(1-w)} -
\log \pi - \gamma - \sum_{n=1}^\infty \( \frac{1}{w+2n}+\frac{1}{1-w+2n}
-\frac1{n} \)+\frac{1}{w(w-1)}.
\]
Now set $w=-\frac12 + iu$.  A short numerical calculation shows that
\[
\max_u \left| -\log \pi - \gamma + \frac{1}{(-1/2+iu)(-3/2+iu)} \right| \le 1.877
\]
and that 
\[
|\zeta'(1-w)/\zeta(w)| \le -\zeta'(3/2)/\zeta(3/2) \le 1.506.
\]
Therefore,
\begin{align*}
\left| \frac{\zeta'(w)}{\zeta(w)} \right| &\le 3.383 + \sum_{n=1}^\infty 
\left| \frac{u^2+n-3/4+2iu} {n(4n^2+2n-3/4+u^2+2iu)} \right| \\
&\le 4.383 + \sum_{n=1}^\infty \frac{n-3/4}{n(4n^2+2n-3/4)} +
\sum_{n=2}^\infty \frac{|u^2+2iu|}{n(4n^2+u^2)} \\
&\le 4.542 + |u|\sqrt{u^2+4} \int_{3/2}^\infty \frac{dx}{x(4x^2+u^2)}  \\
&= 4.542 + \frac{\sqrt{u^2+4}}{2|u|} \log (1+u^2/9) \\
&\le 4.62 + \tfrac12\log(1+u^2/9),
\end{align*}
the last line following from the previous line by another numerical calculation.}
\end{proof}

%
%

\begin{lemma} \label{lem:sum rho}
We have
$$
\sum_{\rho} \frac{1}{|\rho|^2} \le 0.0463,
$$
where the sum is over all of the non-trivial zeros of $\zeta(s)$.
\end{lemma}

\begin{proof} By (\cite{D}, Ch. 9, (10) an (11)), we have
$$
\sum_{\rho} \frac{\Re \rho}{|\rho|^2} = 1 + \tfrac12 \gamma - \tfrac12
 \log(4\pi).
$$
If $\zeta(\rho)=0$ then $\zeta(1-\rho)=0$, and 
\rev{$\Re \rho=1/2$ for $|\Im \rho| \le 5.45\cdot 10^8$.}
Thus
\begin{align*}
\sum_{\rho} \frac{1}{|\rho|^2} &= \sum_{\rho} \( \frac{\Re \rho}{|\rho|^2}
+\frac{\Re\rho}{|1-\rho|^2} \) \\
&\le \rev{2.0001} \sum_{\rho} \frac{\Re \rho}{|\rho|^2} \le 0.0463. \qedhere
\end{align*}
\end{proof}

%
%

\begin{lemma}\label{lem:integral}
Let us fix $\sigma \in [ \tfrac12, 1)$, and suppose
for all $t\ge 3$ we have
\begin{equation}
|\zeta(\sg+iy)| \le X t^Y (\log t)^Z \qquad (1 \le |y| \le t),
\label{eq:zeta upper gen}
\end{equation}
where $X$, $Y$ and $Z$ are positive constants with $Y+Z \ge 0.1$.
If $0<a\le \frac12$,
$t\ge 100$ and $\frac12 \le \sg \le 1-1/t$, then
$$
\int_{-\infty}^\infty \frac{\log |\zeta(\sit+iau)|}{\cosh^2 u}\, du \le
2(\log X + Y\log t + Z \log\log t).
$$
\end{lemma}

\begin{proof}  First, there is no difficulty if $\zeta(\sigma+it+iau)=0$
for some points along the path of integration.  Since all zeros have 
finite order, the integral in the lemma always converges.
 When $-\frac{2t}{a} \le u\le  \frac{-t-1}{a}$,
(3.1) gives $|\zeta(\sit+iau)| \le  X t^Y (\log t)^Z$.
For $\frac{-t-1}{a} \le u \le \frac{-t+3}{a}$, we use the
identity (\cite{T}, (2.1.4))
$$
\zeta(s)=\frac{1}{s-1}+\frac12 + s \int_1^\infty \frac{\lfloor x \rfloor-x
+1/2}{x^{s+1}}\, dx.
$$
Writing $s=\sigma+it+iau$, it follows that $|s-1| \ge 1/t$ and
$|s| \le \sqrt{10}$ and thus $\log|\zeta(s)| \le \log(t+4)$ for this
range of $u$.
For $u\ge \frac{3-t}{a}$, we use the inequalities
$\log(1+x) \le x$ and $\log(1+x)\le x-\frac12 x^2 + \frac13
x^3$, both valid for all $x>-1$.  Then
\begin{align*}
\log |&\zeta(\sit+iau)| \le \log X + Y\log(t+au)+Z\log\log(t+au)\\
 &\le \log(Xt^Y(\log t)^Z)
+\(Y+\frac{Z}{\log t}\) \(\frac{au}{t}-\frac{(au)^2}{2t^2}+\frac{(au)^3}
{3t^3} \).
\end{align*}
Similarly, using $\log(1+x)\le x$, for $u\le - \frac{2t}{a}$
$$
\log |\zeta(\sit+iau)| \le \log(Xt^Y(\log t)^Z)
+\(Y+\frac{Z}{\log t}\) \pfrac{-au-2t}{t}.
$$
Combining these estimates together with
$\int_{-\infty}^\infty (\cosh u)^{-2}\, du = 2$ yields
$$
\int_{-\infty}^\infty \frac{\log |\zeta(\sit+iau)|}{\cosh^2 u}\, du \le
2(\log X + Y\log t + Z \log\log t)  + E,
$$
where
\[
E = \frac{4\log(t+4)}{a\cosh^2 \pfrac{3-t}{a}}
+ \(Y+\frac{Z}{\log t}\) \(  \int_{-\infty}^{-\frac{2t}{a}}
  \frac{-au-2t}{t \cosh^2 u}\, du + \int_{\frac{3-t}{a}}
  ^\infty \frac{\frac{au}{t}-\frac{(au)^2}{2t^2}+\frac{(au)^3}{3t^3}}
  {\cosh^2 u}\, du \).
\]
Now $\frac14 e^{2|u|} \le \cosh^2 u \le e^{2|u|}$, $a\le \frac12$
and $t\ge 100$.  Hence 
$$
ae^{(t-6)/a} \ge 2e^{2t-12}.
$$
Therefore
\begin{align*}
E &\le \frac{16\log(t+4)}{a e^{2(t-3)/a}} + \(Y+\frac{Z}{\log t}\)
  \biggl( \tfrac{4a}{t} e^{-4t/a} \int_0^\infty ve^{-2v}\, dv \\
&\qquad\quad + \int_{-\infty}^\infty  \frac{\frac{au}{t}-\frac{(au)^2}{2t^2}
  +\frac{(au)^3}{3t^3}}{\cosh^2 u}\,du + \int_{\frac{t-3}{a}}^\infty 
  \frac{\frac{au}{t}+\frac{(au)^2}{2t^2} +\frac{(au)^3}{3t^3}}{\frac14
  e^{2u}}\, du \biggr) \\
&\le \frac{32\log(t+4)}{e^{t/a+2t-12}} +  \(\!Y+\frac{Z}{\log t}\!\)
  \(\! \frac{e^{-4t/a}}{t} - \frac{\pi^2 a^2}{12 t^2} + \frac{8a^3}{t^3} 
  \int_{2t-6}^\infty u^3 e^{-2u}\, du \!\) \\
&\le e^{-t/a} + \(Y+\frac{Z}{\log t}\) \( e^{-4t/a} -\frac{\pi^2}{12}
  \frac{a^2}{t^2} + 48 a^3 e^{-4t+12} \) \\
&\le e^{-t/a} - \frac{0.1}{\log t} \frac{a^2}{2t^2} \\
&\le 0. \qquad \text{}
\end{align*}
\end{proof}

%
\section{Detecting zeros of $\zeta(s)$} \label{sec:zeta detect}   %

 From now on, $\rho$ will denote a zero of $\zeta(s)$
and in summations over the zeros, each zero is counted according to
its multiplicity.  Since $\zeta(s) = \overline{\zeta(\bar{s})}$,
when proving zero-free regions
we restrict our attention to the upper half plane.  

%
%

\begin{lemma}\label{lem:real zeta}
Suppose \eqref{eq:Richert} holds.
Let $s = \sigma+it$, \rev{$0<\eta<\pi/4$}, $\sg - \eta \ge 1/2$,  $1\le \sigma \le 1+\eta$
and $t \ge 100$. If $S$ is any subset of $\{ z : \sigma-\eta \le \Re
z \le 1 \}$, then
\begin{align*}
- \Re \, \frac{\zeta'(s)}{\zeta(s)} &\le -  \sum_{\rho \in S, \zeta(\rho)=0} 
\Re \, \frac{\pi}{2\eta}  \cot\pfrac{\pi (s-\rho)}{2\eta} \\
&\qquad +\frac{1}{2\eta} \( \frac23 \log\log t +
  B (1-\sg+\eta)^{3/2} \log t +\log A\) \\ 
&\qquad  - \frac{1}{4\eta} \int_{-\infty}^\infty
  \frac{\log |\zeta(s+\eta+2\eta i u /\pi)|}{\cosh^2 u}\, du.
\end{align*}
\end{lemma}

\begin{proof}  We apply Lemma \ref{lem:zero detect}
 with $f=\zeta$ and $z_0=s$, noting that
$\zeta(0)\ne 0$, all
zeros have real part $<1$ and that $\Re \cot z \ge 0$ for $0 \le \Re z \le
\tfrac{\pi}2$.  Thus the right side in the conclusion of Lemma 
\ref{lem:zero detect} is increased
if we omit from the sum any subset of the zeros. 
Then we apply \eqref{eq:Richert} and Lemma \ref{lem:integral} (with
$X=A$, $Y=B(1-\sigma+\eta)^{3/2}$, $Z=2/3$, $a=2\eta/\pi$) to the integral 
over the line $\Re z = \sigma-\eta$.
Note also that the integral on the right side in Lemma \ref{lem:real zeta}
always converges by Lemma \ref{lem:zeta basic} \rev{(i)}.
Therefore, if $\eta$ is ``bad'' with respect to Lemma \ref{lem:zero detect},
we can apply the above argument
with a sequence of numbers $\eta'$ tending to $\eta$ from above. 
\end{proof}

We next require an upper bound on the number of zeros close to a point
$1+it$.  Here $N(t,R)$ denotes the number of zeros $\rho$ with
$|1+it-\rho|\le R$.

%
%

\begin{lemma}\label{lem:N(t,R)}
Assume \eqref{eq:Richert} holds with $A>1$ and $B>0$.
Then, for $0<R\le 1/4$, $t\ge 100$,
$$
N(t,R) \le 1.3478 R^{3/2} B\log t  + \rev{3.752} + \frac{\log A
-\log R+\tfrac23 \log\log t}{1.879}.
$$
\end{lemma}

\begin{proof}
Apply Lemma \ref{lem:real zeta} with  $s=1+0.6421R+it$, $\eta=2.5R$ 
(so that $\sigma-\eta\ge 
\frac12$) and $S=\{z: |1+it-z|\le R, \Re z \le 1 \}$.
These parameters were chosen to minimize the first term on the right side
of the inequality in the lemma.
By Lemma \ref{lem:zeta basic}, if $v$ is real then
\begin{equation}
\begin{split}
\left| \rev{\frac{\zeta'}{\zeta}(s)} \right| &\le \rev{\frac{1}{0.6421R}}, \\
\rev{\log} |\zeta(s+\eta+iv)|^{-1} &\le \rev{\log \, \zeta(1+3.1421 R) \le \log\bigg(0.64
+ \frac{1}{3.1421R}\bigg).}
\end{split}
\label{eq:s+eta+iv}
\end{equation}
Next, in the region $U=\{ z: \Re z \ge 0.6421, |z-0.6421|\le 1\}$,
we prove
\begin{equation}
\Re \frac{\pi}{5} \cot \pfrac{\pi z}{5} \ge 0.3758. \label{eq:cot1}
\end{equation}
By the maximum modulus principle, it suffices to prove \eqref{eq:cot1}
on the boundary of $U$.  Using
$$
\Re \cot(x+iy) = \frac{2\sin(2x)}{e^{2y}+e^{-2y}-2\cos(2x)},
$$
the minimum of $\Re \cot(x+iy)$ on the vertical segment $x=0.6421\pi/5$,
$|y| \le \pi/5$ occurs at the endpoints.  On the semicircular part of the
boundary of $U$, we verified \eqref{eq:cot1} by
a short computation using the computer algebra package Maple.
In particular, the relative 
minima on the boundary of $U$ occur at $z=1.6421$ and $z=0.6421\pm i$.
Therefore, by \eqref{eq:s+eta+iv},
\eqref{eq:cot1} and Lemma \ref{lem:real zeta},
\[
\rev{-\frac{1}{0.6421R}} \le  -0.3758 \frac{N(t,R)}{R} + \frac{1}{5R}
\biggl( \frac23 \log\log t +  (1.8579R)^{3/2}B\log t + 
 \log A + \rev{\log \big(0.64+\tfrac{1}{3.1421R}\big) } \biggr).
\]
Since \rev{$\log(0.64+\frac{1}{3.1421R}) = -\log R + \log(0.64R+1/3.1421) \le -\log R-0.7376
$}, the lemma follows.
\end{proof}

{\bf Remark.}  A qualitatively similar result
 may also be proved, in a similar way, from
Lemma 2 of \cite{HB1}, or from Landau's lemma (\S 3.9 of \cite{T}).
\bigskip

%
%

\begin{lemma}\label{lem:zeros near 1+it}
Suppose $t\ge 10000$,
$0<v\le 1/4$, and \eqref{eq:Richert} holds with $A>1$, $B>0$.  Then
\begin{multline*}
\sum_{|1+it-\rho|\ge v} \frac{1}{|1+it-\rho|^2} 
\le (6.132 + 5.392 B (v^{-1/2}-2)) \log t \rev{-38.77}\\ -8.5\log A 
+ 4 \log\log t + \frac{\frac{\log A-\log v+\tfrac23\log\log t}{1.879}
+\rev{3.486} - N(t,v)}{v^2}.
\end{multline*}
\end{lemma}

\begin{proof}  Divide the zeros with $|1+it-\rho|\ge v$ into three sets:
\begin{align*} 
Z_1 &= \{ \rho: |\Im \rho - t| \ge 1 \}, \\
Z_2 &= \{ \rho \not\in Z_1: |1+it-\rho| \ge \tfrac14 \text{ and } 
    |it-\rho| \ge \tfrac14 \},\\
Z_3 &= \{ \rho: \rho \not\in Z_2, \rho\not\in Z_1\text{ and } |1+it-\rho|
    \ge v \}.
\end{align*}
For $i=1,2,3$, let $S_i$ be the sum over $\rho \in Z_i$ of $|1+it-\rho|^{-2}$.
By Theorem 19 of \cite{Ro}, the number, $N(T)$, of nontrivial
zeros of $\zeta(s)$ with imaginary part in $[0,T]$ satisfies
\begin{equation}
N(T) = \frac{T}{2\pi} \log \frac{T}{2\pi} - \frac{T}{2\pi} + \frac78 + Q(T),
\label{eq:N(T)}
\end{equation}
where
$$
|Q(T)| \le 0.137\log T + 0.443\log\log T + 1.588 \qquad (T\ge 2).
$$
Since there are no zeros $\rho$ with $|\Im \rho| \le 14$, 
$$
S_1 \le \int_{t+1}^\infty \frac{d N(u)}{(u-t)^2} +
 \int_{14}^{t-1} \frac{d N(u)}{(t-u)^2} +
 \int_{14}^{\infty} \frac{d N(u)}{(u+t)^2} =I_1 + I_2 + I_3.
$$
Since $dN(u)=\frac1{2\pi} \log \frac{u}{2\pi}+dQ(u)$,
$\log(t+x)\le \log t + \frac{x}{t}$ and $\log\log(t+x)\le \log\log t
+\frac{x}{t\log t}$, we have
\begin{align*}
I_1 &\le \frac{1}{2\pi} \int_1^\infty \frac{\log(t+x)-\log 2\pi}{x^2}\, dx + 
  |Q(t+1)| + 2\int_1^\infty \frac{|Q(t+x)|}{x^3}\, dx \\
&= \frac{\( 1 + \tfrac{1}{t} \) \log(1+t)-\log (2\pi)}{2\pi}
  + |Q(t+1)| + 2\int_1^\infty \frac{|Q(t+x)|}{x^3}\, dx \\
&\le 0.4332\log t + 0.886\log\log t + 2.884 +
  2\int_{1}^\infty \frac{0.1851x/t}{x^3}\, dx \\
&\le 0.4332\log t + 0.886\log\log t + 2.885.
\end{align*}
Similarly, noting that $Q(14) \ge 0$, we get
\begin{align*}
I_2 &\le \frac{1}{2\pi} \log\pfrac{t}{2\pi}+2\max_{14\le u\le t-1} |Q(u)| \\
& \le 0.4332\log t+ 0.886\log\log t + 2.884
\end{align*}
and
$$
I_3 \le \frac{1}{2\pi} \int_{14}^\infty \frac{\log\pfrac{u+t}{2\pi}}
{(u+t)^2}\, du + 2 \int_{14}^\infty \frac{|Q(u)|}{(u+t)^3} \, du 
\le 0.00014.
$$
Thus
\begin{equation}
S_1 \le 0.8664\log t + 1.772\log\log t+5.77.
\label{eq:S1}
\end{equation}
Next let $N_2=|Z_2|$ and $N_3=|Z_3|$.
By \eqref{eq:N(T)},
\be
\begin{split}
N_2+N_3 &= N(t+1)-N(t-1)-N(t,v) \\
&\le 0.59231\log t + 0.886\log\log t+2.591-N(t,v).
\end{split}
\label{eq:N2N3}
\end{equation}
In the sum $S_2$, each zero on the critical line contributes $\le 4$
and each pair of zeros $\rho=\beta+i\gamma$, $\rho'=1-\beta+i\gamma$
with $\beta>1/2$ contributes at most $4^2+(4/3)^2$ to the sum.  Therefore,
$$
S_2 \le \frac{80 N_2}{9}.
$$
For $S_3$, $N(t,1/4)$ of the zeros contribute at most
 $(4/3)^2$ each, since $N_3+N(t,v)=2N(t,1/4)$.
By partial summation,
\begin{align*}
S_3 &\le \frac{16N(t,1/4)}{9} + \int_v^{1/4} \frac{dN(t,u)}{u^2} \\
&= \frac{160}{9}N(t,1/4) - \frac{N(t,v)}{v^2} +
2\int_{v}^{1/4} \frac{N(t,u)}{u^3}\, du \\
&= \frac{80N_3}{9} + \( \frac{80}{9} - \frac{1}{v^2} \) N(t,v) +
2\int_{v}^{1/4} \frac{N(t,u)}{u^3}\, du.
\end{align*}
By Lemma \ref{lem:N(t,R)},
\begin{multline*}
 2\int_{v}^{1/4} \frac{N(t,u)}{u^3}\, du \le
\( \frac{\log A+\tfrac23\log\log t}{1.879}+\rev{3.752} \) \( v^{-2} - 16 \) \\
 + 5.3912 B \( v^{-1/2}-2 \)\log t +
 \frac{1}{1.879} \( 8-16\log 4 - \frac{1+2\log v}{2v^2} \).
\end{multline*}
Therefore, using \eqref{eq:N2N3}, we obtain
\begin{align*}
S_2 + S_3 &\le (5.2650+5.3912 B (v^{-1/2}-2))\log t +2.2\log\log t-8.5\log A\\
&\quad \rev{-44.54} + \frac{1}{v^2} \( \frac{\log A-\log v+\tfrac23\log\log t}
{1.879}+\rev{3.486}\) - \frac{N(t,v)}{v^2}.
\end{align*}
Combining this with \eqref{eq:S1} gives the lemma.
\end{proof}

%
%

\begin{lemma}\label{lem:cot}
Suppose that $\Re z \ge 0$ and $|z| \le \pi/2$.   Then
$$
\Re \( \cot z - \frac{1}{z} + \frac{4z}{\pi^2} \) \ge 0.
$$
\end{lemma}

\begin{proof}  By the maximum modulus principle 
it suffices to prove the inequality on the boundary of the region.
On the vertical segment $z=iy$, $-\pi/2 \le y\le \pi/2$, the left side
is zero.  When $|z|=\pi/2$, $z=x+iy$ and $x\ge 0$, the left side is
$$
\frac{2\sin(2x)}{e^{2y}+e^{-2y}-2\cos(2x)} - \frac{x}{x^2+y^2} + \frac{4x}
{\pi^2} = \frac{2\sin(2x)}{e^{2y}+e^{-2y}-2\cos(2x)} \ge 0.
$$
This proves the lemma.
\end{proof}

%
%

The next two lemmas are related to Heath-Brown's method for detecting
zeros from \cite{HB2}.  These give bounds for
a ``mollified'' sum, similar to Lemmas 5.1 and 5.2 of \cite{HB2}.

\begin{lemma}\label{lem:K(s)}
Suppose $f$ is a non-negative real function which has
continuous derivative on $(0,\infty)$. Suppose the Laplace transform
$$
F(z)=\int_0^\infty f(y) e^{-zy}\, dy
$$
of $f$ is absolutely
convergent for $\Re z > 0$.  Let $F_0(z) = F(z) - f(0)/z$ and suppose
\begin{equation}
|F_0(z)| \le  \frac{D}{|z|^2} \qquad (\Re\, z \ge 0, |z| \ge \eta),
\label{eq:F_0}
\end{equation}
where $0<\eta \le \frac32$.  If $\Re s > 1$ and $\Im s \ge 0$, then
\begin{align*}
K(s) &:= \sum_{n=1}^\infty \Lambda(n) n^{-s} f(\log n) \\
&= -f(0) 
\frac{\zeta'(s)}{\zeta(s)} - \sum_\rho F_0(s-\rho)+ F_0(s-1) + E,
\end{align*}
where $|E| \le D(1.72 + \tfrac13 \log(1+\Im s))$.
\end{lemma}

\begin{proof}  We follow the proof of Lemma 5.1 of \cite{HB2}.  Suppose
$s=\sigma+it$ and $1 < \a  < \sigma$.  Define
$$
I = \frac{1}{2\pi i} \int_{\a-i\infty}^{\a+i\infty}-\frac{\zeta'(w)}{\zeta(w)}
F_0(s-w)\, dw.
$$
Since $-\zeta'(w)/\zeta(w)=\sum_n \Lambda(n) n^{-w}$, the sum converging
 uniformly on
$\Re w=\a$, we may integrate term by term.  Thus $I=\sum_n \Lambda(n) J_n$,
where
$$
J_n = \frac{1}{2\pi i}  \int_{\a-i\infty}^{\a+i\infty} n^{-w}
F_0(s-w)\, dw 
= \frac{n^{-s}}{2\pi i}  \int_{\sigma-\a-i\infty}^{\sigma-\a+i\infty} n^u
F_0(u)\, du.
$$
The integral on the right converges absolutely by \eqref{eq:F_0}.  Since
$$
F_0(z) = \frac1{z} \int_0^\infty e^{-zy} f'(y)\, dy,
$$
we have
\begin{align*}
J_n &= \frac{n^{-s}}{2\pi i} \int_0^\infty f'(y) 
 \int_{\sigma-\a-i\infty}^{\sigma-\a+i\infty} \frac{(ne^{-y})^u}{u}\,
du\, dy \\
&=  n^{-s} \int_0^{\log n} f'(y)\, dy =
 n^{-s} \( f(\log n)-f(0) \).
\end{align*}
Thus
\begin{equation}
I = K(s) +f(0) \frac{\zeta'(s)}{\zeta(s)}. \label{eq:I1}
\end{equation}
Moving the line of integration to $\Re w = -1/2$, we have
\begin{equation}
I = \frac{1}{2\pi i} \int_{-1/2-i\infty}^{-1/2+i\infty} 
- \frac{\zeta'(w)}{\zeta(w)} F_0(s-w)\, dw
- \sum_\rho F_0(s-\rho) + F_0(s-1). \label{eq:I2}
\end{equation}
By \eqref{eq:F_0} and Lemma \ref{lem:zeta on -1/2},
 the integral in \eqref{eq:I2} is $\le \frac{D}{2\pi}I'$, where
\begin{align*}
I' &\le \int_{-\infty}^\infty \frac{4.62+\tfrac12 \log(1+u^2/9)}
  {9/4+(u-t)^2}\, du \\
&= 3.08\pi + \frac13 \int_{-\infty}^\infty \frac{\log(1+(t/3+v/2)^2)}{1+v^2}\,
   dv \\
&\le 3.08\pi + \frac13 \int_{-\infty}^\infty \frac{\log(1+t^2)+\log(1+v^2)}
  {1+v^2}\, dv \\
&\le 10.8 + \frac{2\pi \log(1+t)}{3}.
\end{align*}
The lemma now follows from \eqref{eq:I1} and \eqref{eq:I2}.
\end{proof}

{\bf Remarks.}  Examples of functions $f$ satisfying the conditions
of Lemma \ref{lem:K(s)}
are those with compact support (say $[0,x_0]$) and with $f''$ 
continuous and bounded on $(0,x_0)$.
These are the functions considered in \cite{HB2}.
To see that \eqref{eq:F_0}
holds, apply integration by parts twice, noting that
$f(x_0)=f'(x_0)=0$.  This gives
$$
F_0(z) = z^{-2} \( f'(0^+) + \int_0^{x_0} e^{-zt} f''(t)\, dt\).
$$

%
%

\begin{lemma}\label{lem:K(s) II}
Suppose $0<\eta \le \frac12$ and \eqref{eq:Richert} holds with $A>1$, $B>0$.
Let $f$ have compact support
and satisfy \eqref{eq:F_0}.  Suppose $s=1+it$ with $t\ge 1000$.
Then
\begin{align*}
\Re K(s) \le  - \sum_{|1+it-\rho| \le \eta} &\Re \left\{
   F(s-\rho) + f(0) \( \frac{\pi}{2\eta} \cot\pfrac{\pi(s-\rho)}{2\eta}
   -\frac{1}{s-\rho} \)\! \right\}\\
&+ \frac{f(0)}{2\eta} \biggl[\frac{2 \log\log t}3 + B \eta^{3/2} \log t+\log A
  - \frac12 \int_{-\infty}^\infty\! \frac{\log|\zeta(s+\eta+\tfrac{2\eta ui}
  {\pi})|}{\cosh^2 u}\, du \biggr] \\
&\qquad +D \( 1.8 + \frac{\log t}{3} + \sum_{|1+it-\rho| \ge \eta} \frac{1}
  {|1+it-\rho|^2} \).
\end{align*}
In addition,
$$
K(1) \le F(0) + 1.8 D.
$$
\end{lemma}

\begin{proof} Suppose that $\sigma>1$.  By Lemma \ref{lem:K(s)},
$$
K(\sigma) \le -f(0) \frac{\zeta'(\sigma)}{\zeta(\sigma)} + 
F_0(\sigma-1) + 1.72D + D\sum_{\rho} \frac{1}{|1-\rho|^2}.
$$
Since $\zeta(\rho)=0$ implies $\zeta(1-\rho)=0$, we may replace $|1-\rho|^2$
by $|\rho|^2$ in the last sum.  Using Lemmas \ref{lem:zeta basic} \rev{(i)} and 
\ref{lem:sum rho}, we obtain
\begin{equation}
\begin{split}
K(\sigma) &\le \frac{f(0)}{\sigma-1} + F_0(\sigma-1)+1.8D \\
&=F(\sigma-1)+1.8D. 
\end{split}
\label{eq:K(sigma)} 
\end{equation}
When $t\ge 1000$ and $s=\sigma+it$,
$\Re F_0(s-1) \le |F_0(s-1)| \le Dt^{-1} \le 0.001 D$.
Also by \eqref{eq:F_0},
$$
\sum_{|1+it-\rho|>\eta} |F_0(s-\rho)| \le D \sum_{|1+it-\rho|>\eta}
\frac{1}{|1+it-\rho|^2}.
$$
Therefore, combining Lemma \ref{lem:real zeta} (with $S=\{ z: \Re z\le 1,\
 |\rev{\sigma}+it-z|\le \eta\}$) and Lemma \ref{lem:K(s)} gives
\begin{equation}
\begin{split}
\Re K(s) \le  - \sum_{|\rev{\sigma}+it-\rho| \le \eta}\nts\nts &\Re \left\{\!
   F(s-\rho) + f(0) \( \frac{\pi}{2\eta} \cot\pfrac{\pi(s-\rho)}{2\eta}
   -\frac{1}{s-\rho} \)\! \right\}  \\
&+ \frac{f(0)}{2\eta} \biggl[\frac23 \log\log t + B \eta^{3/2} \log t+\log A
  - \frac12 \int_{-\infty}^\infty \! \frac{\log|\zeta(s+\eta+\tfrac{2\eta ui}
  {\pi})|}{\cosh^2 u}\, du \biggr]  \\
&\qquad +D \( 1.8 + \frac{\log t}{3} + \sum_{|1+it-\rho| > \eta} \frac{1}
  {|1+it-\rho|^2} \).
\end{split} \label{eq:Re K(s)} 
\end{equation}
Since $f$ has compact support,  $K(s)$ and $F(s)$
are both entire functions.  Also, on the right side of \eqref{eq:Re K(s)},
$|\log| \zeta(\alpha+i\beta)|| \le |\log \zeta(\alpha)|$ when $\alpha>1$ (by
Lemma \ref{lem:zeta basic} \rev{(i)}).
 Thus we may let $\sigma \to 1^+$ in \eqref{eq:K(sigma)} and 
\eqref{eq:Re K(s)}, and this proves the lemma.
\end{proof} 


%
%
\section{A trigonometric inequality}
%
%
%

We use a trigonometric inequality that is very similar
to what is used in standard treatments.  For any real numbers $a_1,a_2$ we 
have
\begin{equation}
\sum_{j=0}^4 b_j \cos(j\theta)=
8(\cos\theta+a_1)^2(\cos\theta+a_2)^2 \ge 0 \quad (\theta\in \mathbb{R}),
\label{eq:cos}
\end{equation}
where 
\begin{equation}
\begin{split}
b_4&=1, \quad b_3=4(a_1+a_2), \quad b_2=4(1+a_1^2+a_2^2+4a_1a_2), \\
b_1&=(a_1+a_2)(12+16a_1a_2), \quad b_0=b_2-1+8(a_1a_2)^2.
\end{split}\label{eq:b_i}
\end{equation}

\begin{lemma}\label{lem:int sum}
Suppose $a_1, a_2$ are real numbers and define $b_0,\ldots,b_4$ by
\eqref{eq:b_i}. 
Suppose that $\eta>0$ and $t_1$, $t_2$ are real numbers.  Then
$$
\int_{-\infty}^\infty \frac{1}{\cosh^2 u} \sum_{j=1}^4 b_j 
\log \left| \zeta\(1+\eta+ijt_1+iut_2 \) \right|\, du
\ge -2b_0 \log \zeta(1+\eta).
$$
\end{lemma}

{\bf Remark.}
Lemma \ref{lem:int sum}
 marks a departure from other treatments, where the bound
$|\zeta(1+\eta+iw)|\ge \zeta(1+\eta)^{-1}$ is used at the outset (in the
context of a different integral),  which in our
situation gives 
$$
I \ge - 2(b_1+\cdots+b_4)\log \zeta(1+\eta).
$$
The new idea is to combine the $\log |\zeta(\cdot)|$ terms
using \eqref{eq:cos} to significantly reduce this part of the estimation.
The idea in \rev{Lemma \ref{lem:int sum}} accounts for the majority of the improvement 
over Heath-Brown's zero-free region.  See also the remarks at the end
of section \ref{sec:thm2}.
\medskip

\begin{proof}
Denote by $I$ the integral in the lemma.
We begin with the Euler product representation for $\zeta(s)$ in the form
\begin{equation}
\log | \zeta(s) | = -\Re \sum_p \log(1-p^{-s}) = \Re \sum_{\substack{p \\
 m\ge 1}}
 \tfrac1{m} p^{-ms} \quad (\Re s>1). \label{eq:log zeta}
\end{equation}
Next, if $y \ne 0$,
\begin{equation}
U(y) := \int_{-\infty}^\infty \frac{e^{iyu}}{\cosh^2 u} \, du = 
\frac{\pi y}{\sinh (\pi y/2)} \ge 0, \label{eq:U(y)}
\end{equation}
which can be proved by contour integration.
By \eqref{eq:b_i}, \eqref{eq:log zeta} and \eqref{eq:U(y)},
\begin{align*}
I &= \sum_{p,m} \tfrac1{m} p^{-m(1+\eta)} 
  \Re \( \sum_{j=1}^4 b_j p^{-ijmt_1}
  \int_{-\infty}^\infty \frac{p^{-i m u t_2}}{\cosh^2 u} du \) \\
&=  \sum_{p,m} \tfrac1{m}  p^{-m(1+\eta)} U(mt_2 \log p)
\sum_{j=1}^4 b_j \cos(jmt\log p) \\
&\ge -b_0 \sum_{p,m} \tfrac{1}{m} p^{-m(1+\eta)} U(m t_2 \log p).
\end{align*}
Since $U(y) \le 2$ for all $y$, we obtain $I\ge -2b_0\log \zeta(1+\eta)$,
as claimed.
\end{proof}

%
%
\section{The functions $f$, $F$ and $K$}\label{sec:fFK}
%
%
%

Suppose that $t\ge 10000$, $\zeta(\beta+it)=0$ and $\lam$ is
a number with $0<\lam \le 1-\beta$ such that
\begin{equation}
\zeta(s) \ne 0 \qquad
(1-\lam < \Re s \le 1, t-1 \le \Im s \le 4t+1).
\label{eq:lambda}
\end{equation}

Let $f$ be a function with compact support, define
$F$, $F_0$ and $K$ as in Lemma \ref{lem:K(s)}, and assume
that \eqref{eq:F_0} holds.
Let $a_1,a_2$ be real numbers and define $b_0,\ldots,b_4$ by
\eqref{eq:b_i}.  Put $b_5=b_1+b_2+b_3+b_4$. 
By \eqref{eq:cos},
\begin{equation}
\Re \sum_{j=0}^4 b_j K(1+ijt)
= \sum_{n=1}^\infty \Lambda(n) n^{-1} f(\log n) \sum_{j=0}^4 b_j 
\cos(jt\log n) \ge 0. \label{eq:sum K}
\end{equation}
We next apply Lemma \ref{lem:K(s) II} with $s=1$ and $s=1+ijt$ ($j=1,2,3,4$).
Together with Lemma \ref{lem:int sum} (with $t_2=\frac{2\eta}{\pi}$)
and \eqref{eq:sum K}, this gives
\begin{equation}
\begin{split}
0 \le& - \Re \nts\nts\nts 
  \sum_{\substack{1\le j\le 4 \\ |1+ijt-\rho| \le \eta}}
  \nts\nts\nts b_j \( F(1+ijt-\rho) + f(0) \( \tfrac{\pi}{2\eta} \cot\(
  \tfrac{\pi(1+ijt-\rho)}{2\eta}\) -\tfrac{1}{1+ijt-\rho} \)\! \)\\
&+ \frac{f(0)}{2\eta} \biggl[ b_5 \( \tfrac23 L_2 + B \eta^{3/2}L_1 +\log A\)
   + b_0 \log\zeta(1+\eta) \biggr]+b_0F(0) \\
&+ D \biggl(\! b_5 \( 1.8 + \tfrac{L_1}{3}\) + 1.8b_0 +
  \sum_{j=1}^4 b_j \nts \sum_{|1+ijt-\rho| \ge \eta} \frac{1}
  {|1+ijt-\rho|^2}\! \biggr),
\end{split}
\label{eq:big 1}
\end{equation}
where for brevity we write 
$$
L_1=\log(4t+1), \qquad L_2=\log\log(4t+1).
$$

We choose a function $f$ which is based on the functions
given by Lemma 7.5 of \cite{HB2}.  Let $\theta$ be the unique
solution of
\begin{equation}
\sin^2 \theta = \frac{b_1}{b_0} (1-\theta \cot\theta), \quad
0 < \theta < \pi/2,
\label{eq:theta}
\end{equation}
and define the real function
\begin{equation}
g(u) = \begin{cases} 
(\cos(u\tan\theta)-\cos\theta )\sec^2 \theta & |u| \le \frac{\theta}
{\tan\theta}, \\ 0 & \text{else.} \end{cases} \label{eq:g(u)}
\end{equation}
Set $w(u)=g*g(u)$ (the convolution square of $g$) for $u\ge 0$ and
$$
W(z)=\int_0^\infty e^{-zu} w(u)\, du.
$$
From \eqref{eq:g(u)} we deduce (cf. Lemma 7.1 of \cite{HB2}) the identities
\begin{equation}
\begin{split}
W(0) &= 2\sec^2 \theta (1-\theta\cot\theta)^2, \\
W(-1) &= 2\tan^2\theta+3-3\theta(\tan\theta+\cot\theta), \\
w(0) &= \sec^2\theta (\theta\tan\theta+3\theta\cot\theta-3).
\end{split}
\label{eq:Ww}
\end{equation}
Then we take (see \eqref{eq:lambda})
\begin{equation}
f(u) = \lam e^{\lam u} w(\lam u) \qquad (u\ge 0)
\label{eq:f}
\end{equation}
and
\begin{equation}
F(z) = \int_0^\infty e^{-zu} f(u)\, du = W\( \frac{z}{\lam} -1 \).
\label{eq:F}
\end{equation}
For real $y$,
$$
\Re W(iy) = 2 \biggl( \int_0^\infty \rev{g(u)} \cos(uy)\, du \biggr)^2 \ge 0.
$$
Since $W(z) \to 0$ uniformly as $|z|\to \infty$ and $\Re z \ge 0$,
it follows from the maximum modulus principle (applied to $e^{-W(z)}$) that
\begin{equation}
\Re W(z) \ge 0 \qquad (\Re z \ge 0).
\label{eq:Re W}
\end{equation}

%
%
\section{An inequality for the real part of a zero}
\label{sec:zero inequality}
%
%
%

In this section, we take specific values for $a_1$ and $a_2$
and prove the following inequality.

\begin{lemma}\label{lem:zero inequality}
Suppose $t\ge 10000$, $\zeta(\beta+it)=0$ and \eqref{eq:lambda}
holds.  Suppose further that \eqref{eq:Richert} holds with $B>0$ and
\rev{$A \ge 1$}, and that
\begin{equation}
1-\beta \le  \eta/2, \rev{\quad \eta \le 1/4,} \quad 0 < \lam \le \min \( 1-\beta,
\tfrac1{250} \eta \). \label{eq:1-beta}
\end{equation}
Then
\begin{align*}
\frac{1}{\lam} \biggl[ &0.16521 - 0.1876\( \tfrac{1-\beta}{\lam}-1\)
  \biggr]  \le 1.471 \frac{1-\beta}{\eta^2} \\
&+ \frac{1}{2\eta} \left[ \!
  \tfrac{666550}{200211} \(\! \tfrac23 L_2 + B\eta^{3/2}L_1 +\log A \!\) 
  +\log\zeta(1+\eta)\! \right] \\
& + \rev{ 3.476} \lam \!\left[\! (6.466+5.392B(\eta^{-\frac12}-2))L_1 + 4L_2 +
   \frac{\tfrac{\log (A/\eta)+\frac23L_2}{1.879}+\rev{3.486}}{\eta^2} \! \right].
\end{align*}
\end{lemma}

\begin{proof}
A near optimal choice of parameters is $a_1=0.225$, $a_2=0.9$.
By \eqref{eq:b_i},
\begin{align*}
b_0 &= 10.01055 \qquad b_3 = 4.5, \\
b_1 &= 17.14500 \qquad b_4 = 1.0, \\
b_2 &= 10.68250 \qquad b_5 = 33.3275,
\end{align*}
and by \eqref{eq:theta} and \eqref{eq:Ww},
$$
\theta = 1.152214629976363048877\ldots, \quad w(0) = 6.82602968445295450905
\ldots.
$$
The function $W(z)$ has the explicit formula (found with the aid of \rev{Mathematica})
\begin{equation}
W(z) = \frac{w(0)}{z} + W_0(z),
\label{eq:W(z)}
\end{equation}
where
\begin{equation}
W_0(z) = \frac{c_0 \( c_2 ((z+1)^2 e^{-2(\theta/\tan\theta)z} + z^2-1) 
\rev{-c_1 z - c_3 z^3}\)}{z^2(z^2+\tan^2 \theta)^2}
\label{eq:W_0 2}
\end{equation}
and
\begin{align*}
c_0 &= \frac{1}{\sin \theta \cos^3 \theta} = 16.2983216223932350562\ldots \\
c_1 &= \rev{(\theta - \sin\theta\cos\theta)\tan^4 \theta =
   19.9352005926244107856\ldots} \\
c_2 &= \tan^3 \theta \sin^2 \theta = 9.4813169452950521682\ldots \\
c_3 &=  \rev{(\theta - \sin\theta\cos\theta)\tan^2 \theta 
 = 3.945405755634895592\ldots.} 
\end{align*}
If $R \ge 3$, \eqref{eq:W_0 2} implies 
\begin{equation}
|W_0(z)| \le \frac{H(R)}{|z|^3} \qquad (\Re z\ge -1, |z|\ge R),
\label{eq:W_0 bound}
\end{equation}
where
$$
H(R) = \frac{c_0 \( c_2 \tfrac{(R+1)^2}{R^3}\( e^{2\theta/\tan\theta} + 1\) +
\tfrac{c_1}{R^2}+c_3 \)}{\(1-\tfrac{\tan^2\theta}{R^2}\)^2}.
$$
By \eqref{eq:f}, \eqref{eq:F} and \eqref{eq:W(z)},
\begin{align*}
F_0(z) &= F(z) - \frac{f(0)}{z} = W\( \frac{z}{\lam}-1 \) - \frac{\lam
w(0)}{z} \\
&= W_0 \( \frac{z}{\lam}-1 \) + \frac{\lam f(0)}{z(z-\lam)}.
\end{align*}
Suppose $\Re z \ge 0$ and $|z| \ge (R+1) \lam$.  Writing
$z'=\frac{z}{\lam}-1$, we have $\Re z' \ge -1$ and $|z'| \ge R$.  Thus,
by \eqref{eq:f} and \eqref{eq:W_0 bound}, we obtain
$$
|F_0(z)| \le  \frac{H(R)\lam^3}{|z-\lam|^3}+ \frac{w(0)\lam^2}{|z(z-\lam)|}
\le c_4 \frac{\lam f(0)}{|z|^2},
$$
where
\be
c_4 = \frac{H(R) (R+1)^2}{R^3 w(0)} + 1 + 1/R.
\label{eq:c4}
\end{equation}
Therefore, providing that $\eta \ge (R+1)\lam$, \eqref{eq:F_0} holds  with
\begin{equation}
D = c_4 \lam f(0). \label{eq:D}
\end{equation}
Next, define
$$
V_c(z) = c w(0) \( \cot z - \tfrac1{z} \) + W\( \tfrac{z}{c}-1 \).
$$
By \eqref{eq:f} and \eqref{eq:F},
$$
 F(1+ijt-\rho) + f(0) \( \tfrac{\pi}{2\eta} \cot\(
  \tfrac{\pi(1+ijt-\rho)}{2\eta}\) -\tfrac{1}{1+ijt-\rho} \) = V_c(z),
$$
where $z=\frac{\pi}{2\eta}(1+ijt-\rho)$ and $c=\frac{\pi\lam}{2\eta}$.
In order to bound the first double sum in \eqref{eq:big 1}
(leaving only the single term corresponding to
$\rho=\beta+it$), we prove that for $0<c\le \frac{\pi}{2R+2}$,
\begin{equation}
\Re V_c(z) \ge - c_5 c^2 w(0) \qquad 
\(\Re z \ge c, |z| \le \tfrac{\pi}{2}\).,
\label{eq:V_c lower}
\end{equation}
where
\be
c_5 = \frac{4}{\pi^2} \( 1 + \frac{(R+1)^2 H(R)}{w(0) R^3} \) = 
\frac{4}{\pi^2} (c_4 - 1/R).
\label{eq:c5}
\end{equation}
By the maximum modulus principle (applied
to $e^{-V_c(z)}$), it suffices to prove \eqref{eq:V_c lower}
on the boundary of the region.  First consider $z$ satisfying
$\Re z = c$, $|z| \le \pi/2$.  By Lemma \ref{lem:cot} and \eqref{eq:Re W},
$$
\Re V_c(z) \ge c w(0) \Re \( \cot z - \frac{1}{z} \) \ge - \frac{4c^2w(0)}
{\pi^2}.
$$
When $|z|=\pi/2$ and $x=\Re z \ge c$, we have $|z/c-1| \ge R$, so
by \eqref{eq:W_0 bound}, $|W_0(z/c-1)| \le H(R) |z/c-1|^{-3}$.  Thus,
by \eqref{eq:W(z)} and Lemma \ref{lem:cot},
\begin{align*}
\Re V_c(z) &\ge -\frac{4cw(0) x}{\pi^2} + \frac{cw(0)(x-c)}{|z-c|^2} -
\frac{H(R) c^3}{|z-c|^3} \\
&\ge -\frac{4cw(0) x}{\pi^2} + \frac{cw(0)(x-c)}{(\pi/2)^2} -
\frac{H(R) c^3}{(\pi/2-c)^3} \\
&= c^2 w(0) \( -\frac{4}{\pi^2} - \frac{H(R) c}{w(0)(\pi/2-c)^3}\).
\end{align*}
Noting that $c\le \frac{\pi}{2R+2}$ completes the proof of
\eqref{eq:V_c lower}.
In fact, with more work one can prove that \eqref{eq:V_c lower}
 holds with $c_5=\frac13$.

By \eqref{eq:V_c lower}, we have
\begin{equation*}
\begin{split}
 - \Re &\nts\nts\nts 
  \sum_{\substack{1\le j\le 4 \\ |1+ijt-\rho| \le \eta}}
  \nts\nts\nts b_j \( F(1+ijt-\rho) + f(0) \( \tfrac{\pi}{2\eta} \cot\(
  \tfrac{\pi(1+ijt-\rho)}{2\eta}\) -\tfrac{1}{1+ijt-\rho} \)\! \)\\
&\le -b_1V_c(\tfrac{\pi}{2\eta}(1-\beta)) + c_5 c^2 w(0) \sum_{j=1}^4
  b_j N(jt,\eta).
\end{split}
\end{equation*}
Combining this last estimate with \eqref{eq:big 1}, \eqref{eq:f},
\eqref{eq:D} and Lemma \ref{lem:zeros near 1+it} gives
\begin{equation}
\begin{split}
0 \le & \; b_0F(0) -b_1 V_c\( \tfrac{\pi}{2\eta}(1-\beta) \)
   + \frac{\lam f(0)}{\eta^2} \(
   \tfrac{\pi^2c_5}{4} - c_4\) \sum_{j=1}^4 b_j N(jt,\eta) \\
& + \frac{f(0)}{2\eta} \biggl[b_5 \( \tfrac23 L_2 + B \eta^{3/2}L_1 +\log A\)
   + b_0 \log \zeta(1+\eta) \biggr] \\
&  +c_4 \lam f(0) b_5 \biggl[\! 1.8 +\tfrac{L_1}{3}+ 1.8\tfrac{b_0}{b_5} +
    \(\!6.132+5.392B (\eta^{-\frac12}-2)\! \) L_1 \\
&  \rev{- 38.77}-8.5\log A +4L_2 + \frac{1}{\eta^2} 
   \( \frac{\log A-\log\eta+\tfrac23 L_2}{1.879}+\rev{3.486} \) \biggr].
\end{split}
\label{eq:big 2}
\end{equation}
The sum \rev{on $j$} in \eqref{eq:big 2} can be ignored because of \eqref{eq:c5}.
Also, by the lower bound on $A$ we have
\be
1.8+1.8 \tfrac{b_0}{b_5} \rev{- 38.77} - 8.5\log A < 0.
\label{eq:constants}
\end{equation}
Put $R=249$, and compute \rev{$H(249) \le 66.69$ and $c_4\le 1.044$.}
Since $\cot x - \frac{1}{x} \ge -0.348x$ for $0<x\le \frac{\pi}{4}$ and
$1-\beta \le \frac12 \eta$, we have
\begin{equation}
V_c\( \tfrac{\pi}{2\eta}(1-\beta) \) \ge F(1-\beta)-0.348 f(0) \tfrac{\pi^2}
{4\eta^2} (1-\beta). 
\label{eq:V_c 2}
\end{equation}
By \eqref{eq:Ww}, \eqref{eq:f} and \eqref{eq:F},
\begin{equation}
\begin{split}
-\frac{b_1}{b_0} F(1-&\beta) + F(0)= - \( \frac{b_1}{b_0}
W\( \tfrac{1-\beta}{\lam}-1\) - W(-1) \) \\
&= - \( \frac{b_1}{b_0}W(0)-W(-1) \) + \frac{b_1}{b_0} \( W(0) - 
W\( \tfrac{1-\beta}{\lam}-1\) \) \\
&= \frac{-f(0)\cos^2\theta}{\lam}  + \frac{b_1}{b_0} \( W(0) - 
W\( \tfrac{1-\beta}{\lam}-1\) \).
\end{split}
\label{eq:F 2}
\end{equation}
Since $W(x)$ and \rev{$-W'(x)$} are both decreasing, we have
$$
W(0) - W\( \tfrac{1-\beta}{\lam}-1\) \le \( \tfrac{1-\beta}{\lam}-1\)\rev{(- W'(0))}
\le 0.7475 \( \tfrac{1-\beta}{\lam}-1\). 
$$
Thus, by \eqref{eq:V_c 2} and \eqref{eq:F 2},
\begin{equation}
\begin{split}
F(0)-\frac{b_1}{b_0} V_c\( \tfrac{\pi}{2\eta}(1-\beta) \) &\le
0.348 f(0) \frac{\pi^2}{4\eta^2} \frac{b_1}{b_0} (1-\beta) \\
&+ \quad \frac{f(0)}{\lam} \( -\cos^2 \theta
+ \frac{0.7475b_1}{b_0 w(0)} \( \tfrac{1-\beta}{\lam}-1\) \).
\end{split}
\label{eq:beta 1}
\end{equation}
Dividing both sides of \eqref{eq:big 2} by $b_0 f(0)$ and using
\eqref{eq:constants}, \eqref{eq:beta 1} and the numerical values of
$b_0,b_1,b_5$ and $\theta$ completes the proof of the lemma.
\end{proof}

%
%
\section{The proof of Theorem 2}\label{sec:thm2}
%
%
%

Suppose $T_0$ satisfies the hypotheses of Theorem \ref{thm2} and let 
\begin{equation}
M = \inf_{\substack{\zeta(\beta+it)=0 \\ t\ge T_0}} Z(\beta,t), \quad
Z(\beta,t) := (1-\beta) (B\log t)^{\frac23} (\log\log t)^{\frac13}. 
\label{eq:M}
\end{equation}
By the Korobov-Vinogradov theorem, $M>0$.
If $M \ge M_1$, then the theorem is immediate.  Otherwise,
suppose that $M < M_1 \le 0.05507$.  Then there is a zero $\beta+it$
of $\zeta(s)$ with $t\ge T_0$ and
$$
Z(\beta,t) \in [M,M(1+\delta)], \quad
\delta=\min\(\tfrac{10^{-100}}{\log T_0}, \tfrac{M_1-M}{2M} \).
$$
By \eqref{eq:M}, \eqref{eq:lambda} holds with
\begin{equation}
\lam = M L_1^{-2/3} L_2^{-1/3} B^{-2/3}. 
\label{eq:lambda 2}
\end{equation}
Again we make the abbreviations $L_1=\log(4t+1)$, $L_2=\log\log(4t+1)$.
Define $b_0,b_5$ as in the previous section.  We
apply Lemma \ref{lem:zero inequality}, taking
\begin{equation}
\eta = E B^{-\frac23} \pfrac{L_2}{L_1}^{\frac23}, \quad 
E = \pfrac{4(1+b_0/b_5)}{3}^{\frac23} = \pfrac{1733522}{999825}^{\frac23}. 
\label{eq:eta}
\end{equation}
The lower bound $\frac{\log T_0}{\log\log T_0} \ge \frac{1740}{B}$ ensures
that $\eta \le 0.01$ and 
$$
\lam \le \rev{0.05507} (BL_1)^{-\frac23}L_2^{-\frac13} \le \frac{\eta}{250}.
$$
The inequalities $T_0 \ge e^{30000}$ and
$M_1 \le 0.05507$ ensure that the other hypotheses of Lemma 
\ref{lem:zero inequality} are met.
In addition,
\begin{equation}
\frac{1-\beta}{\lam}-1 \le (1+\delta)\pfrac{L_1}{\log t}^{\frac23} \pfrac{L_2}
{\log\log t}^{\frac13}-1 \le \frac{0.97}{\log T_0}.
\label{eq:1-beta lower}
\end{equation}
\rev{Since $\eta \le 0.01$,} by Lemma \ref{lem:zeta basic} \rev{(ii)},
\be
\rev{\log \zeta(1+\eta) \le \log(1/\eta+0.64) \le \log(1/\eta)+0.0064.}
\label{eq:log(1+eta)}
\end{equation}
We now apply Lemma \ref{lem:zero inequality},
 using the upper bounds for $(1-\beta)$ and
$\lam$ on the right side of the conclusion.
First, since $-\log\eta \approx \frac23 L_2$,
we have by \eqref{eq:eta},
\begin{equation}
\begin{split}
\frac{1}{2\eta} \left[ \frac{b_5}{b_0} \( \!\frac{2L_2}{3}+B\eta^{\frac32}
  L_1\!\) +\frac{2L_2}{3} \right] &= \frac{b_5}{b_0}\(1+\frac{b_0}{b_5}
  \)^{\frac13} \pfrac{3B}{4}^{\frac23} L_1^{\frac23} L_2^{\frac13} \\
&\le 2.99968 (BL_1)^{\frac23} L_2^{\frac13}.
\end{split}
\label{eq:main term}
\end{equation}
This constitutes the main term as $t\to \infty$.  Next, since
$Z(\beta,t)\le M_1$ and by the lower bound on $T_0$,
\begin{equation}
1.471 \frac{1-\beta}{\eta^2} \le 0.039 B^{2/3} L_1^{2/3} L_2^{-5/3}
\le 0.0038 B^{2/3} \pfrac{L_1}{L_2}^{2/3}.
\label{eq:error term 1}
\end{equation}
\rev{Recall that $\log T_0 \ge 30000$, thus $L_2 \ge 10.3089 > B$
and $\log(B/L_2)<0$.}
Using \eqref{eq:log(1+eta)},
the remaining part of the second line in the conclusion of Lemma
 \ref{lem:zero inequality} is
\begin{equation}
\begin{split}
& \le\frac{1}{2\eta} \left[ \tfrac{b_5}{b_0}\log A - \log E + 
  \tfrac23 \log(B/L_2) + \rev{0.0064}\right]  \\
&\le \frac{B^{\frac23}}{2E} \pfrac{L_1}{L_2}^{\frac23}
  \left[ \rev{\tfrac{b_5}{b_0}}\log A-\rev{0.36048}+\tfrac23 \log (B/L_2) \right] \\
&\le \pfrac{BL_1}{L_2}^{\frac23} \( 1.1534\log A-\rev{0.1248}+\rev{0.2309}\log(B/L_2) \).
\end{split}
\label{eq:error term 2}
\end{equation}
By \eqref{eq:lambda 2}, \eqref{eq:eta}, and $L_2 \le 0.00035L_1$,
 the third line in the conclusion of Lemma \ref{lem:zero inequality} is
\begin{equation}
\begin{split}
&\le \rev{0.1915} L_1^{-\frac23} L_2^{-\frac13} B^{-\frac23} \Biggl[ \biggl( 6.468+
  \frac{5.392 B^{\frac43}}{\sqrt{E}} \pfrac{L_1}{L_2}^{\frac13} - 10.784 B
   \biggr)L_1 \\
&\qquad\quad + \frac{B^{\frac43}}{E^2} \pfrac{L_1}{L_2}^{\frac43} \( 
  \frac{\log A +\tfrac43 L_2 + \tfrac23 \log(B/L_2)-\log E}{1.879}+\rev{3.486} \)
   \! \Biggr] \\
&\le \pfrac{BL_1}{L_2}^{\frac23}\biggl[ \rev{0.9248}+\frac{\rev{1.239-2.065 B}}
  {B^{\frac43}} \pfrac{L_2}{L_1}^{\frac13} \\
&\qquad\quad  + \frac{\rev{0.04893}}{L_2}
  \( \log A + \tfrac23\log (B/L_2)+\rev{6.18331} \) \biggr] \\
&\le \(\frac{BL_1}{L_2}\!\)^{\!\frac23}\biggl[ \frac{\rev{1.239-2.065 B}}{B^{4/3}}
  \(\frac{L_2}{L_1}\!\)^{\frac13}\!\!+\rev{0.0048 \log A +0.0031
  \log(\tfrac{B}{L_2})}+\rev{0.9542}\!\biggr]
\end{split}
\label{eq:error term 3}
\end{equation}

Combining \eqref{eq:1-beta lower}--\eqref{eq:error term 3}
with Lemma \ref{lem:zero inequality} gives
$$
\frac{1}{\lam} \( 0.16521 - \frac{0.182}{\log T_0}\) \le
 (BL_1)^{\frac23} L_2^{\frac13} \( 2.99968 + \tfrac{X(t)}{L_2} \).
$$
By \eqref{eq:lambda 2}, this gives
$$
M \ge \frac{0.16521-0.182/\log T_0}{2.99968+X(t)/\log\log t} \ge M_1.
$$
This concludes the proof of Theorem \ref{thm2}.
\bigskip

{\bf Remarks.}
Compared with the methods in \cite{HB1}, there are two improvements
evident in \eqref{eq:main term}.
First, the factor $(3/4)^{2/3} \approx 0.82548$ 
replaces the factor $2^{-1/3}K_2 \approx 0.843445$ from
 (\cite{HB1}, p. 197).  This improvement
comes from integrating over two vertical lines (Lemma \ref{lem:zero detect}).
The second
and larger improvement is the factor $(1+b_0/b_5)^{1/3}$, which is
$2^{1/3}$ in the treatment of \cite{HB1}, and comes from  combining the
$\log |\zeta(\cdot)|$ terms in Lemma \ref{lem:int sum}.
  Together these improve the
bounds from \cite{HB1} by about $17\%$.
\bigskip

%
%
\section{The proof of Theorem 3}
%
%
%

Almost everything in Sections \ref{sec:zero detect}--\ref{sec:fFK} is
identical.  
In place of \eqref{eq:Richert} we use
an explicit form of the Van der Corput bound \eqref{eq:CG}.
We fix $\eta=\frac12$, and the proof of Lemma \ref{lem:int sum}
gives
\be
\begin{split}
-\int_{-\infty}^\infty \frac{\sum_{j=1}^4 b_j 
\log |\zeta\(\tfrac32+ijt+\tfrac{iu}{\pi} \)|}{\cosh^2 u}& \, du \le
b_0 \sum_{p,m} \tfrac1{m} p^{-\frac32 m} U(\tfrac{m}{\pi} \log p) \\
&= 2b_0 \sum_{p,m} \frac{\log p}{p^{2m}-p^{m}} \\
&= 2b_0 \sum_{n=2}^\infty \frac{\Lambda(n)}{n^2-n} \\
&\le 1.702b_0.
\end{split}\label{eq:int num spec}
\end{equation}
Let $T_0=545000000$ and suppose that $\zeta(\beta+it)=0$ with $t\ge T_0$
(it is known that all zeros with $|t|<T_0$ have real part $\frac12$
\cite{LRW}).
In place of Lemma \ref{lem:integral} we use

\begin{lemma} \label{lem:clas int}
If $t\ge T_0$, then
$$
I(t) = \int_{-\infty}^{\infty} \frac{\log | \zeta(1/2+it+iu/\pi)|}{\cosh^2 u}
\, du \le 2 J(t),
$$
where $J(t)$ is given by \eqref{eq:J(t)}.
\end{lemma}

\begin{proof}
From \eqref{eq:CG}, $|\zeta(1/2+it)| \le 3t^{1/6}\log t$ for $t\ge 3$, so by
Lemma \ref{lem:integral}, $I(t) \le 2(\frac16 \log t + \log\log t +\log 3)$.
Using the first inequality from \eqref{eq:CG}, we have
\begin{equation*}
I(y) \le \int_{-\infty}^\infty \frac{\log(57+6(t+|u|)^{1/4})}{\cosh^2 u}\,
du = 2\int_0^\infty  \frac{\log(57+6(t+u)^{1/4})}{\cosh^2 u}\, du.
\end{equation*}
When $0\le u \le \log t$, the numerator is $\le \log(6.37306t^{1/4})$ and
when $u>\log t$, the numerator is $\le \log(6.4(e^u+u)^{1/4}) \le u$ and
the denominator is $\ge \frac14 e^{2u}$.  Therefore,
\begin{equation*}
I(t) \le 2\log(6.37306t^{1/4})+8\int_{\log t}^\infty ue^{-2u}\, du \le
\frac{\log t}{2}+3.7042.\qedhere
\end{equation*}
\end{proof}

 We make
the assumption \eqref{eq:lambda} as before and take the same values for
$a_1,a_2$ (so $b_0,\ldots,b_4$, $\theta$, $w$, $f$, $F$, $W$ are the same
as in section \ref{sec:zero inequality}).
The only change in \eqref{eq:big 1}
is that the term \rev{$\frac23 L_2 +B\eta^{3/2} L_1 + \log A$} is replaced by $J(t)$.
Next, we follow the proof of Lemma \ref{lem:zero inequality}.
Using \eqref{eq:N(T)} (Rosser's theorem) as in the proof of Lemma 
\ref{lem:zeros near 1+it}, we obtain for $t\ge 10000$ \rev{and $1\le j\le 4$}
\begin{equation}
\sum_{|1+ijt-\rho|\ge \frac12} \frac{1}{|1+ijt-\rho|^2} \le
3.2357\rev{L_1} + 5.316\rev{L_2} 
 + 16.134-4\rev{N(jt,1/2)}.
\label{eq:sum zero 2}
\end{equation}
\rev{Indeed, the bound in \eqref{eq:S1} handles terms with $|\Im \rho - jt|\ge 1$,
the remaining terms contribute $4(N(t+1)-N(t-1)-N(jt,1/2))$ and we use the bound 
\eqref{eq:N2N3}.}
Assume that
\be
0 < \lambda \le 1-\beta \le \frac{1}{160}.
\label{eq:lambda beta 2}
\end{equation}
Let $R=\frac{1}{2(1-\beta)}-1 \ge 79$.  By \eqref{eq:lambda beta 2},
$\eta \le 80\lambda$.  As in the proof of \eqref{eq:c4}, we deduce
that \eqref{eq:F_0} holds with
\be
D = c_4 \lambda f(0), \quad c_4=\frac{H(79)(R+1)^2}{R^3 w(0)}+1+\frac{1}{R}
\le 1.35.
\label{eq:Dc_4}
\end{equation}
Also, \eqref{eq:V_c lower} is replaced by
\be
\Re V_c(z) \ge -c_5 c^2 w(0) = -c_5 \pi^2 \lambda f(0) \quad
(\Re z \ge c, |z|\le \pi/2),
\label{eq:V_c clas}
\end{equation}
valid for $0<c \le \pi(1-\beta)$ with
\be
c_5 = \frac{4}{\pi^2} + \frac{\pi(1-\beta) H(79)}{w(0) (\pi/2-\pi(1-\beta))^2}.
\label{eq:c_5}
\end{equation}
Analogously to \eqref{eq:big 2}, the inequalities \eqref{eq:int num spec},
\eqref{eq:sum zero 2},
\eqref{eq:Dc_4}, and \eqref{eq:V_c clas} give
\be
\begin{split}
0 &\le b_0 F(0)-b_1 V_{\pi \lam} (\pi(1-\beta)) + (\pi^2 c_5-4c_4) \lam f(0)
\sum_{j=1}^4 b_j N(jt,\tfrac12) \\
&\quad + f(0) \( b_5 J(4t+1) + 0.851b_0\) \\
&\quad + 1.35 \lam f(0) \left[ b_5(1.8+\tfrac{L_1}3)+1.8b_0+b_5(3.2357 L_1 +
 5.316 L_2+16.134) \right].
\end{split}
\label{eq:big 3}
\end{equation}
As before we use $L_1=\log(4t+1)$, $L_2=\log\log(4t+1)$.
By \eqref{eq:Dc_4} and \eqref{eq:c_5}, $\pi^2 c_5-4c_4=-4/R<0$, so the
sum in \eqref{eq:big 3} can be ignored.  By \eqref{eq:lambda beta 2},
$\cot x - 1/x \ge -0.3334 x$ for $0<x\le \pi(1-\beta)$ and this gives
$$
V_{\pi\lam}(\pi(1-\beta)) \ge F(1-\beta)-0.3334\pi^2 (1-\beta) f(0).
$$
By an argument similar to that leading to \eqref{eq:beta 1}, we obtain
\[
F(0)-\frac{b_1}{b_0} V_{\pi\lam}(\pi(1-\beta)) \le 0.3334\pi^2 (1-\beta)
\frac{b_1}{b_0} f(0)
 + \frac{f(0)}{\lam} \(-\cos^2\theta+\frac{0.7475b_1}
{b_0 w(0)} \( \frac{1-\beta}{\lam}-1\)\)
\label{eq:beta 4}
\]
Combining \eqref{eq:beta 4} with \eqref{eq:big 3} gives the following bound.

%
%

\begin{lemma}\label{lem:zero inequality clas}
Suppose that $\zeta(\beta+it)=0$ with $t\ge 545000000$ and $1-\beta\le 
\tfrac{1}{160}$.  Let $\lam$ be a positive number satisfying \eqref{eq:lambda}.
 Then
\begin{multline}
\frac{0.16521-0.1876(\tfrac{1-\beta}{\lam}-1)}{\lam} \le 5.646(1-\beta)
   +\frac{b_5}{b_0} J(4t+1) \\
+ 0.851+ 1.35\lam \frac{b_5}{b_0} (3.5691L_1+5.316L_2+18.475).
\label{eq:last lem}
\end{multline}
\end{lemma}

To prove Theorem \ref{thm3}, first define \rev{$M_0=0.675$, $c_7=0.155$ and }
\begin{equation}\label{c6}
\rev{c_6(t) = 0.05635 \, \frac{J(4t+1)-J(t)+c_7(L_2-\log\log t)}{J(t)+c_7\log\log t+M_0},}
\end{equation}
\rev{where again $L_1=\log(4t+1)$ and $L_2=\log L_1$.}
For a zero $\beta+it$ of $\zeta$ with $t\ge T_0$, define $Y(\beta,t)$ by
the equation
$$
1-\beta = \rev{\frac{0.04962-c_6(t)}{J(t) + c_7\log\log t+ Y(\beta,t)}}.
$$
By the Korobov-Vinogradov theorem \rev{(e.g., Theorem \ref{thm1})}, $Y(\beta,t) \to -\infty$ as $t\to \infty$.
Let $M=\max_{t\ge T_0} Y(\beta,t)$. 
\rev{If $M\le M_0$ then we are done.  Now suppose that $M > M_0$.}
Suppose $\beta+it$ is a zero with $Y(\beta,t)=M$.
\rev{In particular, since $t\ge T_0$ we have $1-\beta \le \frac{1}{160}$.}
\rev{Since $c_6(t)$ is decreasing for $t\ge T_0$,} \eqref{eq:lambda} holds with
\be
\lam =  \rev{\frac{0.04962-c_6(t)}{J(4t+1)+c_7 L_2+M}}.
\label{eq:lambda 3}
\end{equation}
By \eqref{eq:lambda 3},
\begin{equation}
\begin{split}
\frac{1-\beta}{\lam}-1  &= \rev{\frac{J(4t+1)+c_7L_2+M}{J(t)+c_7\log\log t+M}-1
=\frac{J(4t+1)-J(t)+c_7(L_2-\log\log t)}{J(t)+c_7\log\log t+M}}.
\label{eq:beta lam 4}
\end{split}
\end{equation}
Apply Lemma \ref{lem:zero inequality clas}, multiplying both sides of
\eqref{eq:last lem} by \rev{$b_0/b_5$}.  By \eqref{eq:lambda 3}, the left side is
$$
\ge \frac{\rev{0.04962-c_6(t)}}{\lam} = \rev{J(4t+1)+c_7 L_2+M}.
$$
\rev{We conclude that}
\be
\begin{split}
\rev{M+c_7L_2} &\le 0.25562+(1-\beta) \left[ 1.696+1.35(3.5691L_1+5.316L_2+18.475) \right] \\
&= \rev{0.25562+\frac{(0.04962-c_6(t))\big( 1.696+1.35(3.5691L_1+5.316L_2+18.475) \big)}{J(t)+c_7\log\log t+M_0}.}
\end{split}
\label{eq:MM}
\end{equation}
\rev{A computer calculation reveals that $M<0.675$, a contradiction.}

\bigskip
{\bf Thanks.}
The author thanks D. R. Heath-Brown, D. Kutzarova, 
O. Ramar\'e and the referee for helpful suggestions and comments.
The author also thanks
Y. Cheng and S. W. Graham for pre-prints of their work.

\rev{The author thanks Tanmay Khale and Matthew Fernando for
finding mistakes in the original published version, and
thanks Ghaith Hiary, Dhir Patel and Tim Trudgian for helpful
discussions surrounding (1.6).}

%
%
%

\end{document}